\documentclass[12pt,a4paper]{article}
\usepackage[utf8]{inputenc}
\usepackage{amsmath}
\usepackage{amsthm}
\usepackage{amsfonts}
\usepackage{amssymb}
\usepackage{graphicx}
\usepackage{subfigure}
\usepackage{tikz}
\usetikzlibrary{calc,arrows,babel,positioning,shapes,patterns,decorations.markings}
\usepackage{enumerate}
\usepackage{enumitem}

\newtheorem{theorem}{Theorem}

\newtheorem{proposition}[theorem]{Proposition}
\newtheorem{corollary}[theorem]{Corollary}

\usepackage{authblk}

\providecommand{\keywords}[1]
{
  \small	
  \textbf{\textit{Keywords:}} #1
}

\title{Dynamic cycles in edge-colored multigraphs\thanks{This research was supported by grants CONACYT FORDECYT-PRONACES/ 39570/2020 and UNAM DGAPA-PAPIIT IN102320. The second author received a fellowship 782239 from CONACYT.} }
\author[1]{Galeana-Sánchez Hortensia}
\author[1]{Vilchis-Alfaro Carlos\thanks{Corresponding author. Email: vilchiscarlos@ciencias.unam.mx}}
\affil[1]{\normalsize Instituto de Matem\'{a}ticas, Universidad Nacional Aut\'{o}noma de M\'{e}xico, Área de la investigación científica, Circuito Exterior, Ciudad Universitaria, 04510 Coyoacán, CDMX, México}
\date{}

\usepackage[left=23mm,right=23mm,top=25mm,bottom=25mm]{geometry}

\begin{document}
\maketitle

\begin{abstract}
Let $H$ be a graph possibly with loops and $G$ be a multigraph without loops. An $H$-coloring of $G$ is a function $c: E(G) \rightarrow V(H)$. We will say that $G$ is an $H$-colored multigraph, whenever we are taking a fixed $H$-coloring of $G$. The set of all the edges with end vertices $u$ and $v$ will be denoted by $E_{uv}$. We will say that $W=(v_0,e_0^1, \ldots, e_0^{k_0},v_1,e_1^1,\ldots,e_1^{k_1},v_2,\ldots,v_{n-1},e_{n-1}^1,\ldots,e_{n-1}^{k_{n-1}},v_n)$, where for each $i$ in $\{0,\ldots,n-1\}$, $k_i \geq 1$ and $e_i^j \in E_{v_iv_{i+1}}$ for every $j \in \{1,\ldots, k_i \}$, is a dynamic $H$-walk iff $c(e_i^{k_i})c(e_{i+1}^1)$ is an edge in $H$, for each $i \in \{0,\ldots,n-2\}$. We will say that a dynamic $H$-walk is a closed dynamic $H$-walk whenever $v_0=v_n$ and $c(e_{n-1}^{k_{n-1}})c(e_0^1)$ is an edge in $H$. Moreover, a closed dynamic $H$-walk is called dynamic $H$-cycle whenever $v_i\neq v_j$, for every $\{i,j\}\subseteq \{0,\ldots,v_{n-1}\}$. In particular, a dynamic $H$-walk is an $H$-walk whenever $k_i=1$, for every $i \in \{0,\ldots,n-1\}$, and when $H$ is a complete graph without loops, an $H$-walk is well known as a properly colored walk. 

In this work, we study the existence and length of dynamic $H$-cycles, dynamic $H$-trails and dynamic $H$-paths in $H$-colored multigraphs. To accomplish this, we introduce a new concept of color degree, namely, the \textit{dynamic degree}, which allows us to extend some classic results, as Ore's Theorem, for $H$-colored multigraphs. Also, we give sufficient conditions for the existence of hamiltonian dynamic $H$-cycles in $H$-colored multigraphs, and as a consequence, we obtain sufficient conditions for the existence of properly colored hamiltonian cycle in edge-colored multigraphs, with at least $c\geq 3$ colors.
  
\end{abstract}

\keywords{
Edge colored multigraph, Dynamic $H$-walks, Hamiltonian dynamic $H$-cycles, Properly colored cycles \vspace{1em}

Mathematics Subject Classification : 05C38, 05C15
}

\section{Introduction}
For basic concepts, terminology and notation not defined here, we refer the reader to \cite{bang} and \cite{bondy}. Throughout this work, we will consider graphs, multigraphs (graphs allowing parallel edges) and simple graphs (graphs with no parallel edges or loops). Let $G$ be a multigraph, $V(G)$ and $E(G)$ will denoted the set of vertices and edges of $G$, respectively.

An $c$-\textit{edge-coloring} of a graph $G$ is defined as a mapping $\phi : E(G) \rightarrow \{1,\ldots,c\}$. We will say that $G$ is an $c$-\textit{edge-colored graph} whenever we are taking a fixed $c$-edge-coloring of $G$. Different kinds of walks have been studied in $c$-edge-coloring graphs, for example, \textit{monochromatic walk} (that is a walk with all the edges of the same color) and \textit{properly colored walk} (that is a walk with no consecutive edges having the same color, including the first and the last, in the case that the walk be closed). Several authors have worked with this concepts, for example, Bang-Jensen, Bellitto and Yeo \cite{PC-supereulerian}; Barát and Sárközy \cite{barat}; Guo, Broersma, Li and Zhang \cite{PC-supereulerian3}; Guo, Li, Li and Zhang \cite{PC-supereulerian2}. Properly colored walks are of interest for theoretical reasons, for example, as a generalization of walk in undirected and directed graphs, see \cite{bang}, as well as, in graph theory application, for example, in genetic and molecular biology \cite{dorninger,dorninger2}, social science \cite{cs}, channel assignment is wireless networks \cite{raw,raw2}.

Several authors have studied the existence and the length of PC trails, PC cycles and PC paths, see \cite{chetwynd,feng,loDirac,lo}. In particular, Grossman and H\"{a}ggkvist \cite{2-Cycles} were the first to study the problem of the existence of alternating cycles in $c$-edge-colored graphs, and they proved Theorem \ref{theo:c-cycle}, for $c=2$. Later, Yeo \cite{c-cycles} proved it for $c\geq 2$.

\begin{theorem}[Grossman and H\"{a}ggkvist \cite{2-Cycles}, and Yeo \cite{c-cycles}]\label{theo:c-cycle}
Let $G$ be a $c$-edge-colored graph, $c\geq 2$, with no PC cycle. Then, $G$ has a vertex $z \in V(G)$ such that no connected component of $G-z$ is joined to $z$ with edges of more than one color. 
\end{theorem} 

Abouelaoualim, Das, Fernandez de la Vega, Karpinski, Manoussakis, Martinhon and
Saad \cite{Multi-cycles} gave degree conditions, sufficient for an edge-colored multigraph to have a PC hamiltonian cycle. 

\begin{theorem}[Abouelaoualim et al. \cite{Multi-cycles}]\label{theo:Multi-cycle}
Let $G$ be a $c$-edge-colored multigraph, such that no two parallel edges have the same color, of orden $n$. If for every $x\in V(G)$, $\delta_i(x)\geq \lceil (n+1)/2 \rceil$, for every $i \in \{1,\ldots,c\}$.
\begin{enumerate}[label=(\alph*)]
\item If $c=2$, then $G$ has a PC hamiltonian cycle when $n$ is even, and a PC cycle of length $n-1$, when $n$ is odd.
\item If $c\geq 3$, then $G$ has a PC hamiltonian cycle.
\end{enumerate}
\end{theorem}

In this paper we will consider the following edge-coloring. Let $H$ be a graph possibly with loops and $G$ be a graph without loops. An \textit{$H$-coloring} of $G$ is a function $c: E(G) \rightarrow V(H)$. We will say that $G$ is an $H$\textit{-colored graph}, whenever we are taking a fixed $H$-coloring of $G$. A walk $W=(v_0,e_0,v_1,e_1,\ldots,e_{k-1},v_k)$ in $G$, where $e_i=v_i v_{i+1}$ for every $i \in \{0,\ldots,k-1\}$, is an \textit{$H$-walk} iff $(c(e_0),a_0,c(e_1),\ldots,c(e_{k-2}),a_{k-2},c(e_{k-1}))$ is a walk in $H$, with $a_i=c(e_i)c(e_{i+1})$ for every $i \in \{0,\dots,k-2\}$. We will say that $W$ is \textit{closed} if $v_0=v_k$ and $c(e_{k-1})c(e_0) \in E(H)$. Notice that if $H$ is a complete graph without loops, then an $H$-walk is a properly colored walk. And moreover, if $H$ is looped graph with no more edges, then an $H$-walk is a monochromatic walk.


The concepts of $H$-coloring and $H$-walks were introduced, for the first time by Linek and Sands in \cite{Linek}, in the context of kernel theory and related topics, see \cite{H-caminos1,H-caminos2,H-caminos3}. In \cite{H-paseos}, Galeana-S\'{a}nchez, Rojas-Monroy, S\'{a}nchez-L\'{o}pez and Villarreal-Vald\'{e}s gave necessary and sufficient conditions for the existence of closed Euler $H$-trails. In \cite{H-Yeo}, Galeana-S\'{a}nchez, Rojas-Monroy and S\'{a}nchez-L\'{o}pez study the existence of $H$-cycle, in $H$-colored graphs, and extended the Theorem \ref{theo:c-cycle}, in the context of $H$-colored graphs.

Ben\'{i}tez-Bobadilla, Galeana-S\'{a}nchez and Hern\'{a}ndez-Cruz \cite{H-dinamicos} introduced a generalization of $H$-walks as follows. They allowed ``lane changes", i.e., they allowed concatenation of two $H$-walks, say $W_1=(x_0,e_0,x_1,\ldots,x_{k-1},e_{k-1},x_k)$ and $W_2=(y_0,f_0,y_1,\ldots,y_{j-1},f_{j-1},y_j)$, as long as, the last edge of the first one and the first edge of the second one satisfy that $x_{k-1}=y_0$ and $x_k=y_1$ (i.e., the edges $e_0$ and $f_0$ lie between the same end vertices and travel in the same direction that means that star in the same vertex and end in the same vertex). As a result, they defined the following concept: Let $G$ be an $H$-colored multigraph, a \textit{dynamic $H$-walk} in $G$ is a sequence of vertices $W=(x_0,x_1,\ldots,x_k)$ in $G$ such that for each $i \in \{0,\ldots,k-2\}$ there exists an edge $f_i=x_ix_{i+1}$ and there exists an edge $f_{i+1}=x_{i+1}x_{i+2}$ such that $c(f_i)c(f_{i+1})$ is an edge in $H$.

When we deal with a multigraph $G$, we will denote by $E_{uv}^G$ the set of all the edges in $G$ with end vertices $u$ and $v$.

For the purpose of this paper, we need a definition and some more notation that will allow us to know the edges belonging to a dynamic $H$-walk. So, we will say that $W=(v_0,e_0^1, \ldots, e_0^{k_0},v_1,e_1^1,\ldots,$ $e_1^{k_1},v_2,\ldots,v_{n-1},e_{n-1}^1,\ldots,e_{n-1}^{k_{n-1}},v_n)$, where for each $i \in \{0,\ldots, n-1\}$, $k_i \geq 1$ and $e_i^j \in E_{v_iv_{i+1}}^G$ for every $j \in \{1,\ldots, k_i \}$, is a \textit{dynamic $H$-walk} iff $c(e_i^{k_i})c(e_{i+1}^1)$ is an edge in $H$, for each $i \in \{0,\ldots,n-2\}$. We will say that a dynamic $H$-walk is a \textit{closed dynamic $H$-walk} whenever $v_0=v_n$ and $c(e_{n-1}^{k_{n-1}})c(e_0^1)$ is an edge in $H$. Notice that if $W$ is a closed dynamic $H$-walk satisfying that $v_1=v_n$ and $e_{n-1}^{k_{n-1}}$ and $e_0^1$ are parallel in $G$, then $W$ can be rewrite as $W=(v_1,e_1^1,...,v_{n-1}=v_0,e_{n-1}^1,\ldots,e_{n-1}^{k_{n-1}},e_0^1,\ldots, e_0^{k_0},v_n=v_1)$, and $W$ is closed (unless $n=1$, i.e., $W$ is of the form $(v_0,e_0,\ldots,e_k,v_1)$, where $k \geq 1$). If $W$ is a dynamic $H$-walk that does not repeat edges (vertices), then $W$ will be called \textit{dynamic} $H$\textit{-trail} (\textit{dynamic $H$\textit{-path}}). If $W$ is closed and not repeat a vertex, except for the first and the last, then $W$ will be called \textit{dynamic} $H$\textit{-cycle}. 

It follows from the definition of dynamic $H$-walk that every $H$-walk in $G$ is a dynamic $H$-walk in $G$. Moreover, if $G$ has no parallel edges, then every dynamic $H$-walk is an $H$-walk.

A motivation for the study of dynamic $H$-walks in $H$-colored multigraphs are their possible applications. For example, suppose that we are working with a communication network, represented by a graph $G$, where each vertex represent a connection point, and an edge between two connection points means that they have a direct link between them. Moreover, each directed link have failure probability (namely risk; such as, damage, attack, virus, blockage, among many others), this failure probability will be represented by a color assigned to that edge. Now, consider a new graph, say $H$, where each vertex of $H$ is one of the color used in the described coloring of the edges of $G$; and we add an edge in $H$ from one color to another whenever such a colors transition is convenient or possible (for example, if transitions with the same probability of failure are forbidden, then $H$ will have no loops). Notice that in the practice, it is possible to send the same message simultaneously over two or more parallel connections. If it is required to send a message from point $A$ to point $B$ through $G$ with the more convenient form, we need to find a dynamic $H$-walk in $G$ from $A$ to $B$.

In addition, multigraphs can model several applied problems in a more natural way than simple graphs, see \cite{Applied-Multigraphs,Applied-Multigraphs2,Applied-Multigraphs3}.

In this work, we study the existence of dynamic $H$-cycles and dynamic $H$-trails, and the length of dynamic $H$-cycles and dynamic $H$-paths in $H$-colored multigraphs. To accomplish this, we introduce a new concept of color degree, namely, the \textit{dynamic degree}, that allows us to extend some classic results, such as, Ore's Theorem, for $H$-colored multigraphs. Also, we give sufficient conditions for the existence of hamiltonian dynamic $H$-cycles in $H$-colored multigraphs with at most one ``lane change", and as a consequence, we obtain sufficient conditions for the existence of PC hamiltonian cycle in $c$-edge-colored multigraphs, with $c\geq 3$. Moreover, we improve the conditions given in Theorem \ref{theo:Multi-cycle} b) for an infinitely family of multigraphs.

\begin{figure}[ptb]
\centering 
	\scalebox{0.8}{\begin{tikzpicture}
	\tikzset{every loop/.style={min distance=15mm,in=120,out=60,looseness=10}}
	\tikzstyle{vertex}=[circle, draw=black]
	
	\node[](G) at (0,2.5){\huge{$G$}};		
		\node[vertex](v1) at (3,2) {};
		\node[above=1mm of v1](Et.u) {$v_1$};
		\node[vertex](v2) at (1,0) {};
		\node[left=1mm of v2](Et.u) {$v_2$};
		\node[vertex](v3) at (3,-2) {};
		\node[below=1mm of v3](Et.u) {$v_3$};
		\node[vertex](v4) at (5,0) {};
		\node[below=1mm of v4](Et.u) {$v_4$};
		\node[vertex](v5) at (7,2) {};
		\node[above=1mm of v5](Et.u) {$v_5$};
		\node[vertex](v6) at (9,0) {};
		\node[right=1mm of v6](Et.u) {$v_6$};
		\node[vertex](v7) at (7,-2) {};
		\node[below=1mm of v7](Et.u) {$v_7$};
		
		\node[](B1) at (-1,-2){};
		\node[](B2) at (0.5,-2){};
		\node[left=1mm of B1](B3) {$B$};
		\node[](G1) at (-1,-2.5){};
		\node[](G2) at (0.5,-2.5){};
		\node[left=1mm of G1](G3) {$G$};
		\node[](R1) at (-1,-3){};
		\node[](R2) at (0.5,-3){};
		\node[left=1mm of R1](R3) {$R$};
		
		\draw[dashed,blue] (B1) -- (B2);
		\draw[red] (R1) -- (R2);
		\draw[dotted,green] (G1) -- (G2);
		
		\path	(v1)edge [bend right=20,left,blue,very thick,dashed] node {\textcolor{black}{$e_1$}} (v2)
					edge [bend left=20,right,green,very thick,dotted] node {\textcolor{black}{$e_2$}} (v2)
					edge [bend right=20,left,blue,very thick,dashed] node {\textcolor{black}	{$e_3$}} (v4)
					edge [bend left=20, right,red,very thick] node {\textcolor{black}{$e_4$}} (v4)
				(v2)edge [bend right=20,left,green,very thick,dotted] node {\textcolor{black}{$e_5$}} (v3)
					edge [bend left=20,right,red,very thick] node {\textcolor{black}{$e_6$}} (v3)
				(v3)edge [bend right=20,right,blue,very thick,dashed] node {\textcolor{black}	{$e_7$}} (v4)
					edge [bend left=20, left,red,very thick] node {\textcolor{black}{$e_8$}} (v4)
				(v4)edge [bend right=20,right,blue,very thick,dashed] node {\textcolor{black}{$e_9$}} (v5)
					edge [bend left=20,left,blue,very thick,dashed] node {\textcolor{black}{$e_{10}$}} (v5)
					edge [bend right=20,left,green,very thick,dotted] node {\textcolor{black}	{$e_{11}$}} (v7)
					edge [bend left=20, right,red,very thick] node {\textcolor{black}{$e_{12}$}} (v7)
				(v5)edge [bend right=20,left,blue,very thick,dashed] node {\textcolor{black}{$e_{13}$}} (v6)
					edge [bend left=20,right,red,very thick] node {\textcolor{black}{$e_{14}$}} (v6)
				(v6)edge [bend right=20,left,red,very thick] node {\textcolor{black}	{$e_{15}$}} (v7)
					edge [bend left=20, right,red,very thick] node {\textcolor{black}{$e_{16}$}} (v7);
		\node[](H) at (11,2.5){\huge{$H$}};
		\node[vertex,fill=blue](B) at (12.5,1) {};
		\node[below=1mm of B]{$B$};		
		\node[vertex,fill=red](R) at (14.5,0) {};
		\node[below=1mm of R] {$R$};
		\node[vertex,fill=green](G) at (12.5,-1) {};
		\node[below=1mm of G] {$G$};
		
		\draw[] (B) -- (R);
		\draw[] (G) -- (R);	

		\end{tikzpicture}}
\caption{The sequence $P=(v_4,e_9,v_5,e_{14},e_{13},v_6,e_{16},v_7,e_{11},e_{12},v_4)$ is a dynamic $H$-cycle in $G$ and there is no $H$-cycle of length greater than 2 containing $v_5$, $v_6$ or $v_7$}
\label{fig:dynamic_edge}
\end{figure}
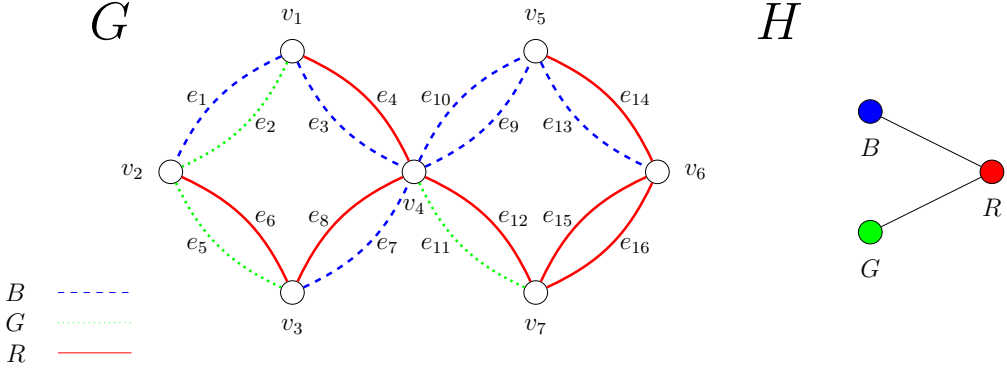

\section{Notation and Terminology}\label{Section:NyT}
Let $G$ be a multigraph. If $e$ is an edge and $u$ and $v$ are the vertices such that $e=uv$, then $e$ is said to \textit{join} $u$ and $v$, we will say that $u$ and $v$ are the ends of $e$, we will say that the edge $e$ is incident with $u$ (respectively $v$) and also we will say that $u$ and $v$ are adjacent. If $u=v$, then the edge $e$ is a loop. The set of all the edges in $G$ with end vertices $u$ and $v$ will be denoted by $E_{uv}^G$, when there is no confusion, for simplicity, we will write $E_{uv}$. Let $e$ and $f$ be two edges in $E_{uv}$, we will say that $e$ and $f$ are parallel. The \textit{neighborhood} of a vertex $u$, denoted by $N_G(u)$, is defined as the set of all the vertices adjacent with $u$ in $G$.

A walk in a multigraph $G$ is a sequence $(v_0, v_1, \ldots , v_k)$, where $v_i v_{i+1} \in E(G)$ for every $i$ in $\{0, \ldots , k-1\}$. We define the length of the walk $W$ as the number $k$, denoted by $l(W)$. We will say that a walk is closed if $v_0 = v_k$. If $v_i \neq v_j$ for all $i$ and $j$ with $i \neq j$, it is called a path. A cycle is a closed walk $(v_0, v_1, \ldots , v_k, v_0)$, with $k \geq 3$, such that $v_i \neq v_j$ for all $i$ and $j$ with $i \neq j$. 

A graph $G$ is said to be multipartite, if for some positive integer $k$, there exists a partition ${X_1, \ldots , X_k }$ of $V(G)$, such that $X_i$ is an independent set in $G$ (that is no two vertices of $X_i$ are adjacent) for every $i$ in $\{1, \ldots , k\}$, in this case, also $G$ is called $k$-partite. It said that $G$ is a complete $k$-partite graph whenever $G$ is $k$-partite and for every $u$ in $X_i$ and for every $v$ in $X_j$ , with $i \neq j$ , we have that $u$ and $v$ are adjacent, denoted by $K_{n_1,\ldots,n_k}$ where $\vert X_i \vert = n_i$ for every $i$ in $\{1, \ldots, k\}$. In the particular case when $k=2$, the graph $G$ is said to be bipartite graph.

Let $G$ be an $H$-colored multigraph, a sequence $W=(v_0,e_0^1, \ldots, e_0^{k_0}, v_1, e_1^1, \ldots,e_1^{k_1},v_2,\ldots,v_{n-1},$ $e_{n-1}^1,\ldots,e_{n-1}^{k_{n-1}},v_n)$ in $G$, where for each $i \in \{0,\ldots, n-1\}$, $k_i \geq 1$ and $e_i^j \in  E_{v_iv_{i+1}}$ for every $j \in \{1,\ldots, k_i \}$, is a \textit{dynamic $H$-walk} in $G$ iff $c(e_i^{k_i})c(e_{i+1}^1)$ is an edge in $H$, for each $i \in \{0, \ldots,n-2\}$. We define length of the dynamic $H$-walk, denoted by $l(W)$, as the number $n$. We will say that the dynamic $H$-walk $W$ has $k_i-1$ \textit{changes} from $v_i$ to $v_{i+1}$, and the number of changes of $W$ is $\sum_{i=1}^{n-1} (k_i-1)$. Notice that if $W$ is a dynamic $H$-walk with zero changes, then $W$ is an $H$-walk. In Figure \ref{fig:dynamic_edge}, $T=(v_1,e_4,v_4,e_{10},v_5,e_{14},e_{13},v_6,e_{15},v_7)$ is a dynamic $H$-trail with one change and $C=(v_4,e_9,v_5,e_{14},e_{13},v_6,e_{15},v_7,e_{11},v_4)$ is a dynamic $H$-cycle with one change.

\section{Main Results}\label{Section:Main}
In what follows $H$ will be a graph possibly with loops, and $G$ will be an $H$-colored multigraph.

We will use an auxiliary graph, denoted by $G_u$, that is defined as follows: Let $G$ be an $H$-colored multigraph and $u$ be a vertex of $G$; $G_u$ is the simple graph such that $V(G_u)=\{e \in E(G): e \textrm{ is incident}\textrm{with } u\}$, and two different vertices $a$ and $b$ are joining by only one edge in $G_u$ if and only if $c(a)$ and $c(b)$ are adjacent in $H$.\vspace{1em}

Let $G$ be an $H$-colored multigraph and $\{u,v\} \subseteq V(G)$. We will say that $E_{uv}$ is a \textit{dynamic edge set} if and only if there exist $\{e,f\} \subseteq E_{uv}$ such that $N_H(c(e)) \neq N_H(c(f))$ and neither of them is subset of the other. The \textit{dynamic degree} of $u$, denoted by $\delta_{dym}(u)$, is the number of vertices $v$ such that $E_{uv}$ is a dynamic edge set. In Figure \ref{fig:dynamic_edge}, the set $E_{v_1v_4}$ is a dynamic edge set but $E_{v_1v_2}$ is not a dynamic edge set, since $N_H(c(e_1)) = N_H(c(e_2))= \{R\}$.\vspace{1em}

\textbf{Observation 1.} If $G_u$ is a complete $k_u$-partite graph, for some $k_u \geq 2$, then $E_{uv}$ is a dynamic edge set if and only if, there exist $e$ and $f$ in $E_{uv}$ in different sets of the partition of $V(G_u)$.

\begin{proposition}\label{proposition:walk}
Let $G$ be an $H$-colored multigraph such that $G_u$ is a complete $k_u$-partite graph, for every $u$ in $V(G)$ and for some $k_u$, $k_u \geq 2$. If $T=(x_0,x_1,\ldots,x_n)$ is a walk in $G$ such that for each $i \in \{0,\ldots,n-1\}$, $E_{x_ix_{i+1}}$ is a dynamic edge set, then there exist $e_i \in E_{x_ix_{i+1}}$, for every $i \in \{0,\ldots, n-1\}$, such that $T'=(x_0,e_0,x_1,\ldots,x_{n-1},e_{n-1},x_n)$ is an $H$-walk. Moreover, if $E_{x_0x_n}$ is a dynamic edge set, then there exist $\{e_n,e_{n+1}\} \subseteq E_{x_0x_n}$ such that $C=(x_0,e_0,x_1,\ldots,x_{n-1},e_{n-1},$ $x_n,e_n,x_0)$ is a closed $H$-walk or $C=(x_0,e_0,x_1,\ldots,x_{n-1},e_{n-1},x_n, e_n,e_{n+1},x_0)$ is a closed dynamic $H$-walk, i.e., there exists a closed dynamic $H$-walk with at most one change.
\end{proposition}
\begin{proof}
Suppose that $G$ is an $H$-colored multigraph such that $G_u$ is a complete $k_u$-partite graph, for every $u$ in $V(G)$ and for some $k_u \geq 2$. Let $T=(x_0,x_1,\ldots,x_n)$ be a walk in $G$ such that $E_{x_ix_{i+1}}$ is a dynamic edge set.

Consider the edge $e_0=x_0x_1$ in $E(G)$, since $E_{x_1x_2}$ is a dynamic edge set, by Observation 1, we have that there exist $f_1=x_1x_2$ and $f_2=x_1x_2$ in different sets of the partition of $V(G_{x_1})$. It follows from the fact that $e_0 \in V(G_{x_1})$ and $G_{x_1}$ is a complete $k_{x_1}$-partite graph, that $e_0f_1 \in E(G_{x_1})$ or $e_0f_2 \in E(G_{x_1})$. And then $e_1$ will be the edge such that $e_0e_1 \in E(G_{x_1})$, i.e., $e_1=f_1$ or $e_1=f_2$ (in case that both edges are adjacent to $e_0$, we take $e_1=f_1$).

Since $E_{x_2x_3}$ is a dynamic edge set, by Observation 1, we have that there exist $g_1=x_2x_3$ and $g_2=x_2x_3$ in different sets of the partition of $V(G_{x_2})$. It follows from the fact that $e_1 \in V(G_{x_2})$ and $G_{x_2}$ is a complete $k_{x_2}$-partite graph, that $e_1g_1 \in E(G_{x_2})$ or $e_1g_2 \in E(G_{x_2})$. And then $e_2$ will be the edge such that $e_1e_2 \in E(G_{x_2})$, i.e., $e_2=g_1$ or $e_2=g_2$ (in case that both edges are adjacent to $e_1$, we take $e_2=g_1$). Repeating this procedure, we have that for every $i \in \{0,\ldots,n-1\}$, there exist $e_i \in E_{x_ix_{i+1}}$ such that $T'=(x_0,e_0,x_1,e_1,x_2,\ldots,x_{n-1},e_{n-1},x_n)$ is an $H$-walk.

Now, suppose that $E_{x_0x_n}$ is a dynamic edge set, since $G_{x_n}$ is a complete $k_{x_n}$-partite graph, then there exists $e_n=x_nx_0$ such that $e_{n-1}e_n \in E(G_{x_n})$. 

If $e_ne_0 \in E(G_{x_0})$, then $C=(x_0,e_0,x_1,\ldots,x_n,e_n,x_0)$ is a closed $H$-walk. 

Otherwise, $e_ne_0 \not\in E(G_{x_0})$. And since $E_{x_nx_0}$ is a dynamic edge set, then there exist an edge $e_{n+1}=x_nx_0$ such that $e_n$ and $e_{n+1}$ are in different sets of the partition of $V(G_{x_0})$. Now, since $G_{x_0}$ is a complete $k_{x_0}$-partite graph and $e_ne_0 \not\in E(G_{x_0})$, then $e_{n+1}e_0 \in E(G_{x_0})$. Therefore, $C=(x_0,e_0,x_1,T',x_n,e_n,e_{n+1},x_0)$ is a closed dynamic $H$-walk.
\end{proof}

\begin{corollary}\label{lemma:tray-ciclo}
Let $G$ be an $H$-colored multigraph such that $G_u$ is a complete $k_u$-partite graph, for every $u$ in $V(G)$ and for some $k_u$, $k_u \geq 2$. If $T=(x_0,x_1,\ldots,x_n)$ is a path in $G$ such that for every $i \in \{0,\ldots,n-1\}$, $E_{x_ix_{i+1}}$ is a dynamic edge set, then there exist $e_i \in E_{x_ix_{i+1}}$, for every $i \in \{0,\ldots, n-1\}$, such that $T'=(x_0,e_0,x_1,\ldots,x_{n-1},e_{n-1},x_n)$ is an $H$-path. Moreover, if $E_{x_0x_n}$ is a dynamic edge set, then there exist $\{e_n,e_{n+1}\} \subseteq E_{x_0x_n}$ such that $C=(x_0,e_0,x_1,\ldots,x_{n-1},e_{n-1},x_n,e_n,x_0)$ is an $H$-cycle or $C=(x_0,e_0,x_1,\ldots,x_{n-1},e_{n-1},x_n, e_n,e_{n+1},x_0)$ is a dynamic $H$-cycle, i.e., there exists a dynamic $H$-cycle with at most one change.
\end{corollary}

\begin{theorem}\label{theorem:dynamic-cycle}
Let $G$ be an $H$-colored multigraph such that $G_u$ a complete $k_u$-partite graph, for every $u$ in $V(G)$ and for some $k_u$, $k_u \geq 2$. If $\delta_{dym}(u) \geq d \geq 2$, for every $u \in V(G)$, then $G$ has a dynamic $H$-cycle of length at least $d+1$ and with at most one change.
\end{theorem}
\begin{proof}
Let $T=(x_0,x_1,\ldots,x_k)$ be a path of maximum length such that $E_{x_ix_{i+1}}$ is a dynamic edge set, for every $i \in \{0,\ldots,k-1\}$.

Claim 1. $T$ has length at least $d$.

Suppose that $T$ has length at most $d-1$. Since $\delta_{dym}(x_k) \geq d$, then there exists a vertex $x_{k+1} \in V(G)\setminus V(T)$ such that $E_{x_kx_{k+1}}$ is a dynamic edge set. Hence, $T'=(x_0,x_1,\ldots,x_k,x_{k+1})$ is a path of length $k+1$ such that $E_{x_ix_{i+1}}$ is a dynamic edge set, for every $i \in \{0,\ldots,k\}$, contradiction. Therefore, $T$ has length at least $d$.

Since $T$ is of maximum length, if $E_{x_0u}$ is a dynamic edge set, then $u \in V(T)$ (otherwise we can extend $T$). Let $j= max\{i \; | \; E_{x_0x_i}$ is a dynamic edge set$\}$. Since $\delta_{dym}(x_0) \geq d$, we have that $j \geq d$. Therefore, $C'=(x_0,x_1,\ldots,x_j,x_0)$ is a cycle such that $E_{x_ix_{i+1}}$ is a dynamic edge set, for every $i \in \{0, \ldots, j\}$, hence by Corollary \ref{lemma:tray-ciclo}, we have that there exist $C$ a dynamic $H$-cycle of length $j+1 \geq d+1$ with at most one change.
\end{proof}

Let $K_n^2$ be a complete multigraph with $|E_{uv}|=2$, for every $\{u,v\} \in V(K_n^2)$. And, let $H$ be a complete simple graph with $k\geq 2$ vertices, and $G$ the union of two $K_n^2$ that share a unique vertex. If we $H$-color $G$ in such a way that every pair of parallel edges has different color, then $\delta_{dym}(x)\geq n-1$, for every $x \in V(G)$, and the length of the maximum dynamic $H$-cycle in $G$ is $n$. So, we cannot improve the length of the dynamic $H$-cycle in Theorem $\ref{theorem:dynamic-cycle}$.

\begin{corollary}
Let $G$ be an $H$-colored complete multigraph such that $G_u$ is a complete $k_u$-partite graph, for every $u \in V(G)$ and for some $k_u$, $k_u \geq 2$. If $E_{xy}$ is a dynamic edge set, for every $\{x,y\} \subseteq V(G)$, then $G$ has a hamiltonian dynamic $H$-cycle. 
\end{corollary}

\begin{theorem}
Let $G$ be an $H$-colored multigraph such that $G_u$ a complete $k_u$-partite graph, for every $u$ in $V(G)$ and for some $k_u$, $k_u \geq 3$. If $\delta_{dym}(u) \geq d \geq 2$, for every $u \in V(G)$, then $G$ has an $H$-path of length at least $\textrm{min}\{2d,n\}$, or $G$ has an $H$-cycle of length at least $d+1$.
\end{theorem}
\begin{proof}
Suppose that $G$ is an $H$-colored multigraph such that for every $u \in V(G)$, we have that $\delta_{dym}(u) \geq d \geq 2$, and $G_u$ is a complete $k_u$-partite graph, for some $k_u \geq 3$.

Then, for every $x \in V(G)$, there exist $e_x,f_x,g_x \in E(G)$ such that $\{e_x,f_x\} \subseteq E_{xv_x}$, for some $v_x \in V(G)$, $g_x=xy_x \not\in E_{xv_x}$ and $e_x,f_x$ and $g_x$ are in different parts of the partition of $G_x$ (it possible by Observation 1 and the fact that $\delta_{dym}(x) \geq 2$).

Let $T=(u_0,u_1,\ldots,u_{j-1}=v_x, u_j=x, u_{j+1}=y_x, \ldots, u_k)$ the longest path such that $\{g_x, e_x\} \subseteq E(T)$ and $E_{u_iu_{i+1}}$ is a dynamic edge set, for every $i \in \{0,\ldots,j-1,j+1,\ldots,k-1\}$, i.e., $E_{xy_x}$ is the only not necessarily a dynamic edge set in $T$.

Notice that $l(T)$ is at least $d$ since $\delta_{dym}(u_0)\geq d$ and $T$ is of maximum length path, hence $k=l(W) \geq d$.

If $k \geq 2d$, since $E_{u_{j-2}u_{j-1}}$ is a dynamic edge set, then there is an edge $e_{j-2} \in E_{u_{j-2}u_{j-1}}$ such that $c(e_{j-2})c(g_x) \in E(H)$. So, by following the same procedure as in the proof of Proposition \ref{proposition:walk}, we can construct the following: 1) An $H$-path from $u_{j-1}$ to $u_0$ starting with the edge $e_{j-2}$, say $T_0=(u_{j-1},e_{j-2},u_{j-2},\ldots,u_1,e_0,u_0)$; and 2) an $H$-path from $u_j$ to $u_k$ starting with the edge $e_x$, say $T_1=(u_j,e_x,u_{j+1},\ldots,u_{k-1},e_k,u_k)$. 

Hence, $T'=(u_0,T_0^{-1},u_{j-1},g_x,u_j,T_1,u_k)$ is an $H$-path of length $k \geq 2d$.

So, suppose that $k \leq 2d-1$.

\textbf{Case 1.} $j+1 \leq d$.

Notice that if $E_{u_1v}$ is a dynamic edge set, then $v \in V(T)$. Otherwise, $T'=(v,u_0,u_1,\ldots,u_k)$ is a path of length $k+1$ such that $E_{v_iv_{i+1}}$ is a dynamic edge set, for every $i \in \{0,\ldots,j-1,j+1,\ldots,k-1\}$, contradicting the choice of $T$. Therefore, $v \in V(T)$.

On the other hand, since $\delta_{dym}(u_0) \geq d$, then there is $u_p \in V(T)$, where $d \leq p \leq k$ such that $E_{u_0u_p}$ is a dynamic edge set. Hence, $C=(u_0,u_1,\ldots,u_j, u_{j+1},\ldots,u_p,u_0)$ is a cycle in $G$ such that $E_{u_iu_{i+1}}$ is a dynamic edge set, for every $i \in \{0,\ldots,j-1,j+1,\ldots,p-1\}$.

Since $C$ is a cycle, we can rewrite it as follows $C=(u_j=x,u_{j-1}=y_x,u_{j-2},\ldots,u_0,u_p,$ $u_{p-1}, \ldots,u_{j+1}=v_x,u_j=x)$. Since $E_{u_{j-2}u_{j-1}}$ is a dynamic edge set, there is an edge $e_{j-2} \in E_{u_{j-2}u_{j-1}}$ such that $c(e_{j-2})c(g_x) \in E(H)$. Hence, by Corollary \ref{lemma:tray-ciclo}, there is an $H$-path $T'=(x,g_x,y_x,e_{j-2},u_{j-2},\ldots,u_{j+2},e_{j+2},u_{j+1}=v_x)$.

Since $e_x$ and $f_x$ are incident with $v_x$, then $e_{j+2}$, $e_x$ and $f_x$ are vertices of $G_{v_x}$. We know that $e_x$ and $f_x$ are in different parts of the partition, then $e_{j+2}e_x \in E(G_{v_x})$ or $e_{j+2}f_x \in E(G_{v_x})$. Hence, $C'=(x,g_x,y_x,e_{j-2},u_{j-3},\ldots, u_{j+2},e_{j+2},$ $v_x,e_x,x)$ or $C'=(x,g_x,y_x,e_{j-2},u_{j-3},\ldots, u_{j+2},e_{j+2},v_x,$ $f_x,x)$ is an $H$-cycle (because, pairwise $e_x,f_x$ and $g_x$ are in different parts of the partition of $G_x$) and $C'$ is of length at least $d+1$.

\textbf{Case 2.} $j+1 > d$.

Notice that if $E_{u_kv}$ is a dynamic edge set, then $v \in V(T)$. Otherwise, $T'=(u_0,u_1,\ldots,u_k,$ $u_{k+1}=v)$ is a path of length $k+1$ such that $E_{v_iv_{i+1}}$ is a dynamic edge set, for every $i \in \{0,\ldots,j-1,j+1,\ldots,k\}$, contradicting the choice of $T$. Therefore, $v \in V(T)$.

On the other hand, since $\delta_{dym}(u_k) \geq d$, then there is $u_p \in V(T)$, where $p \leq 2d-d-1 = d-1$, such that $E_{u_ku_p}$ is a dynamic edge set. Hence, $C=(u_p,u_{p+1},\ldots,u_{j-1},u_j, u_{j+1},\ldots,u_k,u_p)$ is a cycle in $G$ such that $E_{u_iu_{i+1}}$ is a dynamic edge set, for every $i \in \{p,\ldots,k-1\}\setminus\{j\}$.

Since $C$ is a cycle, we can rewrite it as follows $C=(u_j=x,u_{j-1}=y_x,u_{j-2},\ldots,u_p,u_k,$ $u_{k-1}, \ldots,u_{j+1}=v_x,u_j=x)$. Since $E_{u_{j-2}u_{j-1}}$ is a dynamic edge set, there is an edge $e_{j-2} \in E_{u_{j-2}u_{j-1}}$ such that $c(e_{j-2})c(g_x) \in E(H)$. Hence, by Corollary \ref{lemma:tray-ciclo}, there is an $H$-path $T'=(x,g_x,y_x,e_{j-2},u_{j-2},\ldots,u_p,e_k,u_k,e_{k-1},u_{k-1},\ldots,u_{j+2},e_{j+2},u_{j+1}=v_x)$.

Since $e_x$ and $f_x$ are incident with $v_x$, then $e_{j+2}$, $e_x$ and $f_x$ are vertices of $G_{v_x}$. We know that $e_x$ and $f_x$ are in different parts of the partition, then $e_{j+2}e_x \in E(G_{v_x})$ or $e_{j+2}f_x \in E(G_{v_x})$. Hence, $C'=(x,g_x,y,e_{j-2},u_{j-3},T',$ $ u_{j+2}, e_{j+2}, v_x, e_x, x)$ or $C'=(x,g_x,y,e_{j-2},u_{j-3},T', u_{j+2},e_{j+2},v_x,$ $f_x,x)$ is an $H$-cycle (because, pairwise $e_x,f_x$ and $g_x$ are in different parts of the partition of $G_x$) and $C'$ is of length at least $d+1$.
\end{proof}

Let $G$ be an $H$-colored multigraph. We will say that the \textit{dynamic graph of} $G$, denoted by $G_{dym}$, is the simple graph such that $V(G_{dym})=V(G)$ and two different vertices $u$ and $v$ are adjacent, with only one edge, in $G_{dym}$ if and only if $E_{uv}$ is a dynamic edge set in $G$.

\begin{theorem}\label{theo:Dynamic Euler}
Let $G$ be an $H$-colored multigraph such that $G_u$ a complete $k_u$-partite graph, for every $u$ in $V(G)$ and for some $k_u$, $k_u \geq 2$. If $G_{dym}$ is connected and $\delta_{dym}(x)= 2p_x$, where $p_x \geq 1$, then $G$ has a spanning closed dynamic $H$-trail with at most one change.  
\end{theorem}
\begin{proof}
Suppose that $G_{dym}$ is connected and $\delta_{dym}(x)= 2p_x$, where $p_x \geq 1$, then we have that $G_{dym}$ has a closed Euler trail, say $T=(x_0,x_1,\ldots,x_n,x_0)$.

Then, $T$ is a spanning closed dynamic $H$-trail in $G$ such that $E_{x_i x_{i+1}}$ is a dynamic edge set, for every $i \in \{0,1,2,\ldots,n\}$ (when $i=n$, then $x_{i+1}=x_0$). Therefore, by Proposition \ref{proposition:walk}, there exist a spanning closed dynamic $H$-trail in $G$ with at most one change.
\end{proof}

\begin{theorem}\label{theo:Dynamic Ore}
Let $G$ be an $H$-colored multigraph with $n$ vertices such that $G_u$ a complete $k_u$-partite graph, for every $u$ in $V(G)$ and for some $k_u$, $k_u \geq 2$. 
\begin{enumerate}[label=(\alph*)]
\item If $\delta_{dym}(u) + \delta_{dym}(v) \geq n$, for every $\{u,v\} \subseteq V(G)$ such that $E_{uv}$ is not a dynamic edge set, then $G$ has a hamiltonian dynamic $H$-cycle with at most one change.

\item If there is $x_0 \in V(G)$ such that $k_{x_0} \geq 3$, and $\delta_{dym}(x)+\delta_{dym}(y)\geq n+1$, for every $\{x,y\} \subseteq V(G)$, such that $E_{xy}$ is not a dynamic edge set, then $G$ has a hamiltonian $H$-cycle.
\end{enumerate}

\end{theorem}
\begin{proof}

a) Notice that $\delta_{G_{dym}}(u) + \delta_{G_{dym}}(v) \geq n$ for every pair of non adjacent vertices $u$ and $v$ in $G_{dym}$. Hence, by Ore's Theorem, we have that $G_{dym}$ has a hamiltonian cycle, say $C=(x_1,x_2,\ldots, x_n, x_1)$.

Then, $C$ is a hamiltonian cycle in $G$ such that $E_{x_i x_{i+1}}$ is a dynamic edge set, for every $i \in \{1,2,\ldots,n\}$ (when $i=n$, then $x_{i+1}=x_1$). Therefore, by Corollary \ref{lemma:tray-ciclo}, there exist a hamiltonian dynamic $H$-cycle with at most one change.\vspace{1em}

b) Suppose that $G$ is an $H$-colored multigraph such that $G_u$ is a complete $k_u$-partite graph, for some $k_u \geq 2$, for every $u \in V(G)$, and $\delta_{dym}(x)+\delta_{dym}(y)\geq n+1$, for every $\{x,y\} \subseteq V(G)$ such that $E_{xy}$ is not a dynamic edge set and there is $x_0 \in V(G)$ such that $k_{x_0} \geq 3$.

It follows from (a) that $G$ has a hamiltonian cycle such that $E_{x_ix_{i+1}}$ is a dynamic edge set, for every $i \in \{0,\ldots,n\}$ (when $i=n$, then $x_{i+1}=x_0$), say $C=(x_0,x_1,\ldots,x_n,x_0)$.

Notice that $G_{x_i}[E_{x_ix_{i-1}} \cup E_{x_ix_{i+1}}]$ is a complete $k'_{x_i}$-partite graph, where $2 \leq k'_{x_i} \leq k_{x_i}$, for every $i \in \{0,\ldots,n\}$.

\textbf{Case 1.} There exists $i \in \{0,\ldots,n\}$ such that $k'_{x_i} \geq 3$.

Then there exist $e \in E_{x_{i-1}x_i}$, $g \in E_{x_ix_{i+1}}$ and $f \in E_{x_ix_{i-1}} \cup E_{x_ix_{i+1}}$ which are in different parts of $G_{x_i}[E_{x_ix_{i-1}} \cup E_{x_ix_{i+1}}]$.

If $f \in E_{x_ix_{i-1}}$. By Corollary \ref{lemma:tray-ciclo}, there is an $H$-path $T_1=(x_i,g,x_{i+1},\ldots,x_{i-2},e_{i-2},x_{i-1})$. Hence, $C_1=(x_i,T_1,x_{i-1},e,x_i)$ or $C_2=(x_i,T_1,x_{i-1},f,x_i)$ is a hamiltonian $H$-cycle in $G$.

If $f \in E_{x_ix_{i+1}}$. By Corollary \ref{lemma:tray-ciclo}, there is an $H$-path $T_2=(x_i,e,x_{i-1},\ldots,x_{i+2},e_{i+2},x_{i+1})$. Hence, $C_3=(x_i,T_2,x_{i+1},f,x_i)$ or $C_4=(x_i,T_2,x_{i-1},g,x_i)$ is a hamiltonian $H$-cycle in $G$.\vspace{1em}

\textbf{Case 2.} $k'_{x_i} = 2$, for every $x \in V(G)$, i.e., $G_{x_i}[E_{x_ix_{i-1}} \cup E_{x_ix_{i+1}}]$ is a complete bipartite graph.

Let $A=\{g \in V(G_{x_0}) : ge \in E(G_{x_0}) \text{ for every } e \in V(G_{x_0}[E_{x_0x_1} \cup E_{x_0x_n}])\}$. Since $G_{x_0}$ is a complete $k_{x_0}$-partite graph and $k_{x_0} \geq 3$, then $A \neq \emptyset$. Let $p =  max\{i : E_{x_0x_i} \cap A \neq \emptyset\}$. Notice that $p \not\in \{0,1,n\}$ because of the condition of the case and there is no loops.

If $E_{x_1x_{p+1}}$ is a dynamic edge set. By Corollary \ref{lemma:tray-ciclo} and $E_{x_px_{p-1}}$ is a dynamic edge set, there is an $H$-path $T_3=(x_p,e_p,x_{p-1},\ldots, x_1,e_1,x_{p+1},e_{p+1},x_{p+2},\ldots,x_n)$ such that $ge_p \in E(G_{v_p})$. Since $E_{x_nx_0}$ is a dynamic edge set and $g \in A$, we have that there is an edge $e_n \in E_{x_nx_0}$ such that $C_5=(x_0,g,x_1,T,x_n,e_n,x_0)$ is a hamiltonian $H$-cycle in $G$.

If $E_{x_1x_{p+1}}$ is not a dynamic edge set, then by the hypothesis $\delta_{dym}(x_1)+\delta_{dym}(x_{p+1})\geq n+1$. 

\textbf{Subcase 1.} There is $j$, where $2 < j \leq p$, such that $E_{x_1x_j}$ and $E_{x_{p+1}x_{j-1}}$ are dynamic edge sets.

When $j=p$, Corollary \ref{lemma:tray-ciclo} and the fact that $E_{x_1x_p}$ is a dynamic edge set, imply that $T_4=(x_0,g,x_p,e_p,x_1,\ldots,x_{p-1},e_{p-1},x_{p+1},\ldots,x_n)$ is an $H$-path. Since, $E_{x_nx_0}$ is a dynamic edge set and $g \in A$, there is an edge $e_n \in E_{x_nx_0}$ such that  $C_4=(x_0,g,x_p,T_4,x_n,e_n,x_0)$ is a hamiltonian $H$-cycle. 
Otherwise, $T_5=(x_0,g,x_p,e_p,x_{p-1},\ldots,x_j,e_j,x_1,e_1,x_2,\ldots, x_{j-1},e_{j-1},x_{p+1},e_{p+1},x_{p+2},\ldots,x_n)$ is an $H$-path. Since, $E_{x_nx_0}$ is a dynamic edge set and $g \in A$, there is an edge $e_n \in E_{x_nx_0}$ such that $C_5=(x_0,g,x_p,T_5,x_n,e_n,x_0)$ is a hamiltonian $H$-cycle. 
\vspace{1em}

\textbf{Subcase 2.} There is $j$, where $p+2 \leq j \leq n$, such that $E_{x_1x_j}$ and $E_{x_{p+1}x_{j+1}}$ are dynamic edge sets.

When $j=n$, Corollary \ref{lemma:tray-ciclo} and the fact that $E_{x_px_{p-1}}$ is a dynamic edge set, imply that $T_6=(x_0,g,x_p,e_p,x_{p-1},\ldots,x_1,e_1,x_n,e_n,x_{n-1},\ldots,x_{p+1})$ is an $H$-path. Since $E_{x_0x_{p+1}}$ is a dynamic edge set, $g \in A$ and by the maximality of $g$; there is an edge $e_{p+1} \in E_{x_{p+1}x_0}$ such that $C_6=(x_0,T_6,x_{p+1},e_{p+1},x_0)$ is a hamiltonian $H$-cycle. 
Otherwise, $T_7=(x_0,g,x_p,e_p,x_{p-1},\ldots,x_1,e_1,x_j,$ $e_j,x_{j-1},\ldots, x_{p+1},e_{p+1},x_{j+1},e_{j+1},x_{j+2},\ldots,x_n)$ is an $H$-path. Since, $E_{x_nx_0}$ is a dynamic edge set and $g \in A$, there is an edge $e_n \in E_{x_nx_0}$ such that $C_7=(x_0,g,x_p,T_7,x_n,e_n,x_0)$ is a hamiltonian $H$-cycle. 
\vspace{1em}

\textbf{Subcase 3.} For each $j$, where $2 < j \leq p$, at least one of $E_{x_1x_j}$ or $E_{x_{p+1}x_{j-1}}$ is not a dynamic edge set, and for each $k$, where $p+2 \leq k \leq n$, at least one of $E_{x_1x_k}$ or $E_{x_{p+1}x_{k+1}}$ is not a dynamic edge sets.

In this case, $\delta_{dym}(x_{p+1}) \leq (n-2)- (\delta_{dym}(x_1)-2) = n-\delta_{dym}(x_1)$. So, $\delta_{dym}(x_1)+\delta_{dym}(x_{p+1})\leq n$, a contradiction.

Therefore, $G$ has a hamiltonian $H$-cycle.
\end{proof}

We think (but still we cannot prove) that the statement of Theorem \ref{theo:Dynamic Ore}b remains true if we replace the condition $\delta_{dym}(x)+\delta_{dym}(y)\geq n+1$ by $\delta_{dym}(x)+\delta_{dym}(y)\geq n$. Moreover, we cannot replace it by $\delta_{dym}(x)+\delta_{dym}(y)\geq n-1$, since if we $H$-color $G$, the multigraph resulting from the union of two $K_n^3$ that share a unique vertex, in such a way that every pair of parallel edges has different color, where $H$ is a complete simple graph with at least three vertices. Then, $G$ has no hamiltonian $H$-cycle.

\begin{theorem}
Let $G$ be an $H$-colored multigraph with $n$ vertices such that $G_u$ is a complete $k_u$-partite graph, for every $u$ in $V(G)$ and for some $k_u$, $k_u \geq 2$. If $\delta_{dym}(u) + \delta_{dym}(v) \geq n-1$, for every pair of distinct  vertices $u$ and $v$ of $G$ such that $E_{uv}$ is not a dynamic edge set, then $G$ has a hamiltonian $H$-path.
\end{theorem}

\begin{theorem}
Let $G$ be an $H$-colored multigraph with $n$ vertices such that $G_u$ is a complete $k_u$-partite graph, for every $u$ in $V(G)$ and for some $k_u$, $k_u \geq 2$. If $\delta_{dym}(u) + \delta_{dym}(v) \geq n+1$, for every pair of distinct  vertices $u$ and $v$ of $G$ such that $E_{uv}$ is not a dynamic edge set, then for every pair of distinct vertices $x$ and $y$, there is a hamiltonian $H$-path between $x$ and $y$.
\end{theorem}

\begin{corollary}\label{Cor:Dynamic Dirac}
Let $G$ be an $H$-colored multigraph such that $G_u$ a complete $k_u$-partite graph, for every $u$ in $V(G)$ and for some $k_u$, $k_u \geq 2$. If $\delta_{dym}(u) \geq n/2$, for every $u \in V(G)$, then $G$ has a hamiltonian dynamic $H$-cycle with at most one change.
\end{corollary}

\begin{corollary}\label{Cor:Dynamic Dirac 3+}
Let $G$ be an $H$-colored multigraph such that $G_u$ is a complete $k_u$-partite graph, for every $u \in V(G)$, and for some $k_u$, $k_u \geq 3$. If $\delta_{dym}(x) \geq (n+1)/2$, for every $x \in V(G)$, then $G$ has a hamiltonian $H$-cycle.
\end{corollary}

Recall that a $c$-edge-colored multigraph can be represented as an $H$-colored multigraph if $H$ is a complete graph with $c$ vertices and without loops. Moreover, if $\{e,f\} \subseteq E_{xy}$, for some $\{x,y\} \subseteq V(G)$, such that $c(e)\neq c(f)$, i.e., $e$ and $f$ are parallel edges of different color, then $E_{xy}$ is a dynamic edge set, by Observation 1.

\begin{corollary}
Let $G$ be a $c$-edge-colored multigraph such that every vertex is incident to at least two edges of different color. If at least one vertex is incident to at least three edges of different color and, for every pair of distinct vertices $x$ and $y$, $\delta_{dym}(x) + \delta_{dym}(y) \geq n+1$, then $G$ has a PC hamiltonian cycle.
\end{corollary}

\begin{corollary}
Let $G$ be a $c$-edge-colored multigraph such that every vertex is incident to at least two edges of different color. If $\delta_{dym}(x) \geq (n+1)/2$, for every $x \in V(G)$, and at least one vertex is incident to at least three edges of different color, then $G$ has a PC hamiltonian cycle.
\end{corollary}

Recall that in a $c$-edge-colored multigraph, say $G$, we say that $N_i^G(x)$ denotes the set of vertices of $G$ that are joined to $x$ with an edge of color $i$. The $i$th degree of $x$, $x\in V(G)$, denoted by $\delta_i^G(x)$, is equal to $|N_i^G(x)|$, i.e., the cardinality of $N_i^G(x)$. When there is no confusion, for simplicity, we will write $N_i(x)$ and $\delta_i(x)$ instead of $N_i^G(x)$ and $\delta_i^G(x)$, respectively.

Its follows by the definition of $\delta_i^G(x)$ that if $\{e,f\} \subseteq E_{xu}$, for some $u \in V(G)$ such that $e$ and $f$ have the same color, then $\delta_i^{G\setminus \{e\}}(x)=\delta_i^{G\setminus \{f\}}(x)=\delta_i^G(x)$. So, in what follows, we will consider edge-colored multigraphs with no parallel edges with the same color. Therefore, if $G$ is an $c$-edge-colored multigraph, then $|E_{uv}| \leq c$, for every $\{u,v\}\subseteq V(G)$.

\begin{theorem}
Let $G$ be a $c$-edge-colored multigraph, $c\geq 3$, with $n$ vertices and $|E_{uv}|\leq c-1$, for every $\{u,v\}\subseteq V(G)$. If for every $x\in V(G)$, $\delta_i(x) \geq n/2$, for every $i \in \{1,\ldots,c\}$, then $G$ has PC hamiltonian cycle.
\end{theorem}
\begin{proof}
Suppose that $G$ is a $c$-edge-colored multigraph, $c\geq 3$, with $n$ vertices,  $|E_{uv}|\leq c-1$, for every $\{u,v\}\subseteq V(G)$, and $\delta_i(x) \geq n/2$, for every $x\in V(G)$ and for every $i \subseteq \{1,\ldots,c\}$.

\textbf{Claim.} $\delta_{dym}(x)\geq (n+1)/2$, for every $x\in V(G)$.

Proceeding by contradiction, suppose that there is a vertex $u \in V(G)$ such that $\delta_{dym}(u)=k < (n+1)/2$.

On the one hand, since $\delta_i(u) \geq n/2$, for every $i \in \{1,\ldots,c\}$, we have that $d(u)\geq cn/2$, i.e, the number of edges incident to $x$ is at least $cn/2$.

On the other hand, if $E_{uy}$ is a dynamic edge set, then $u$ is joined to $y$ with at most $(c-1)$ edges. Otherwise, $E_{uy}$ is not a dynamic edge set and $u$ is joined to $y$ with at most one edge. Then, $d(u) \leq (n-1-k)+(c-1)k = n-1+(c-2)k$.

Therefore, $d(u) \leq n-1+(c-2)k < n-1+(c-2)(n+1/2)=c(n+1/2)-2 < c(n+1/2) \leq d(u)$, a contradiction.

Therefore, $\delta_{dym}(x)\geq (n+1)/2$, for every $x \in V(G)$, and by Corollary \ref{Cor:Dynamic Dirac 3+}, $G$ has PC hamiltonian cycle.
\end{proof}

We think (but still we cannot prove) that the previous theorem remains true, if we remove the condition ``$|E_{uv}| \leq c-1$, for every $\{u,v\} \subset V(G)$''.

\end{document}